\documentclass[12pt]{amsart}
\usepackage{mathrsfs}      
\usepackage{txfonts}
\usepackage{amssymb}
\usepackage{eucal}
\usepackage{graphicx}
\usepackage{amsmath}
\usepackage{amscd}
\usepackage[all]{xy}           
\usepackage[active]{srcltx} 
\usepackage{amsfonts,latexsym}
\usepackage{xspace}
\usepackage{epsfig}
\usepackage{float}
\usepackage{color}
\usepackage{fancybox}
\usepackage{colordvi}
\usepackage{multicol}
\usepackage{colordvi}
\usepackage[hypertex]{hyperref} 
\newtheorem{theorem}{Theorem}[section]
\newtheorem{prop}[theorem]{Proposition}
\newtheorem{defn}[theorem]{Definition}
\newtheorem{lemma}[theorem]{Lemma}

\newtheorem{prop-def}{Proposition-Definition}[section]
\newtheorem{coro-def}{Corollary-Definition}[section]

\newtheorem{remark}{Remark}[section]

\newcommand{\nc}{\newcommand}
\nc{\tred}[1]{\textcolor{red}{#1}}
\nc{\tblue}[1]{\textcolor{blue}{#1}}
\nc{\tgreen}[1]{\textcolor{green}{#1}}
\nc{\tpurple}[1]{\textcolor{purple}{#1}}
\nc{\btred}[1]{\textcolor{red}{\bf #1}}
\nc{\btblue}[1]{\textcolor{blue}{\bf #1}}
\nc{\btgreen}[1]{\textcolor{green}{\bf #1}}
\nc{\btpurple}[1]{\textcolor{purple}{\bf #1}}

\newcommand{\efootnote}[1]{}

\renewcommand{\textbf}[1]{}

\newcommand{\delete}[1]{}
\nc{\dfootnote}[1]{{}}          
\nc{\ffootnote}[1]{\dfootnote{#1}}
\nc{\mfootnote}[1]{\footnote{#1}} 
\nc{\ofootnote}[1]{\footnote{\tiny Older version: #1}}

\delete{
\nc{\mlabel}[1]{\label{#1}}  
\nc{\mcite}[1]{\cite{#1}}  
\nc{\mref}[1]{\ref{#1}}  
\nc{\mcite}[1]{\cite{#1}{{\bf{{\ }(#1)}}}}  
}

\nc{\mlabel}[1]{\label{#1}  
{\hfill \hspace{1cm}{\bf{{\ }\hfill(#1)}}}}
\nc{\mcite}[1]{\cite{#1}{{\bf{{\ }(#1)}}}}  
\nc{\mref}[1]{\ref{#1}{{\bf{{\ }(#1)}}}}  

\nc{\mbibitem}[1]{\bibitem[\bf #1]{#1}} 
\nc{\mkeep}[1]{\marginpar{{\bf #1}}} 

\nc{\opa}{\ast} \nc{\opb}{\odot} \nc{\op}{\bullet} \nc{\pa}{\frakL}
\nc{\arr}{\rightarrow} \nc{\lu}[1]{(#1)} \nc{\mult}{\mrm{mult}}
\nc{\diff}{\mathfrak{Diff}}
\nc{\opc}{\sharp}\nc{\opd}{\natural}
\nc{\ope}{\circ}
\nc{\opf}{\dashv}
\nc{\opg}{\vdash}
\nc{\oph}{\cdot}
\nc{\bin}[2]{ (_{\stackrel{\scs{#1}}{\scs{#2}}})}  
\nc{\binc}[2]{ \left (\!\! \begin{array}{c} \scs{#1}\\
    \scs{#2} \end{array}\!\! \right )}  
\nc{\bincc}[2]{  \left ( {\scs{#1} \atop
    \vspace{-1cm}\scs{#2}} \right )}  
\nc{\bs}{\bar{S}} \nc{\cosum}{\sqsubset} \nc{\la}{\longrightarrow}
\nc{\rar}{\rightarrow} \nc{\dar}{\downarrow} \nc{\dprod}{**}
\nc{\dap}[1]{\downarrow \rlap{$\scriptstyle{#1}$}}
\nc{\md}{\mathrm{dth}} \nc{\uap}[1]{\uparrow
\rlap{$\scriptstyle{#1}$}} \nc{\defeq}{\stackrel{\rm def}{=}}
\nc{\disp}[1]{\displaystyle{#1}} \nc{\dotcup}{\
\displaystyle{\bigcup^\bullet}\ } \nc{\gzeta}{\bar{\zeta}}
\nc{\hcm}{\ \hat{,}\ } \nc{\hts}{\hat{\otimes}}
\nc{\barot}{{\otimes}} \nc{\free}[1]{\bar{#1}}
\nc{\uni}[1]{\tilde{#1}} \nc{\hcirc}{\hat{\circ}} \nc{\lleft}{[}
\nc{\lright}{]} \nc{\lc}{\lfloor} \nc{\rc}{\rfloor}
\nc{\curlyl}{\left \{ \begin{array}{c} {} \\ {} \end{array}
    \right .  \!\!\!\!\!\!\!}
\nc{\curlyr}{ \!\!\!\!\!\!\!
    \left . \begin{array}{c} {} \\ {} \end{array}
    \right \} }
\nc{\longmid}{\left | \begin{array}{c} {} \\ {} \end{array}
    \right . \!\!\!\!\!\!\!}
\nc{\onetree}{\bullet} \nc{\ora}[1]{\stackrel{#1}{\rar}}
\nc{\ola}[1]{\stackrel{#1}{\la}}
\nc{\ot}{\otimes} \nc{\mot}{{{\boxtimes\,}}}
\nc{\otm}{\overline{\boxtimes}} \nc{\sprod}{\bullet}
\nc{\scs}[1]{\scriptstyle{#1}} \nc{\mrm}[1]{{\rm #1}}
\nc{\margin}[1]{\marginpar{\rm #1}}   
\nc{\dirlim}{\displaystyle{\lim_{\longrightarrow}}\,}
\nc{\invlim}{\displaystyle{\lim_{\longleftarrow}}\,}
\nc{\mvp}{\vspace{0.3cm}} \nc{\tk}{^{(k)}} \nc{\tp}{^\prime}
\nc{\ttp}{^{\prime\prime}} \nc{\svp}{\vspace{2cm}}
\nc{\vp}{\vspace{8cm}} \nc{\proofbegin}{\noindent{\bf Proof: }}
\nc{\proofend}{$\blacksquare$ \vspace{0.3cm}}
\nc{\modg}[1]{\!<\!\!{#1}\!\!>}
\nc{\intg}[1]{F_C(#1)} \nc{\lmodg}{\!
<\!\!} \nc{\rmodg}{\!\!>\!}
\nc{\cpi}{\widehat{\Pi}}
\nc{\sha}{{\mbox{\cyr X}}}  
\nc{\shap}{{\mbox{\cyrs X}}} 
\nc{\shpr}{\diamond}    
\nc{\shp}{\ast} \nc{\shplus}{\shpr^+}
\nc{\shprc}{\shpr_c}    
\nc{\msh}{\ast} \nc{\zprod}{m_0} \nc{\oprod}{m_1}
\nc{\vep}{\varepsilon} \nc{\labs}{\mid\!} \nc{\rabs}{\!\mid}

\nc{\mmbox}[1]{\mbox{\ #1\ }} \nc{\fp}{\mrm{FP}}
\nc{\rchar}{\mrm{char}} \nc{\End}{\mrm{End}} \nc{\Fil}{\mrm{Fil}}
\nc{\Mor}{Mor\xspace} \nc{\gmzvs}{gMZV\xspace}
\nc{\gmzv}{gMZV\xspace} \nc{\mzv}{MZV\xspace}
\nc{\mzvs}{MZVs\xspace} \nc{\Hom}{\mrm{Hom}} \nc{\id}{\mrm{id}}
\nc{\im}{\mrm{im}} \nc{\incl}{\mrm{incl}} \nc{\map}{\mrm{Map}}
\nc{\mchar}{\rm char} \nc{\nz}{\rm NZ} \nc{\supp}{\mathrm Supp}
\nc{\GL}{\mathrm{GL}}
\nc{\Alg}{\mathbf{Alg}} \nc{\Bax}{\mathbf{Bax}} \nc{\bff}{\mathbf f}
\nc{\bfk}{{\bf k}} \nc{\bfone}{{\bf 1}} \nc{\bfx}{\mathbf x}
\nc{\bfy}{\mathbf y}
\nc{\base}[1]{\bfone^{\otimes ({#1}+1)}} 
\nc{\Cat}{\mathbf{Cat}}

\nc{\detail}{\marginpar{\bf More detail}
    \noindent{\bf Need more detail!}
    \svp}
\nc{\Int}{\mathbf{Int}} \nc{\Mon}{\mathbf{Mon}}
\nc{\rbtm}{{shuffle }} \nc{\rbto}{{Rota-Baxter }}
\nc{\remarks}{\noindent{\bf Remarks: }} \nc{\Rings}{\mathbf{Rings}}
\nc{\Sets}{\mathbf{Sets}} \nc{\wtot}{\widetilde{\odot}}
\nc{\wast}{\widetilde{\ast}} \nc{\fopb}{\bar{\odot}}
\nc{\fopa}{\bar{\ast}} \nc{\hodot}[1]{\odot^{#1}}
\nc{\hast}[1]{\ast^{#1}} \nc{\mal}{\mathcal{O}}
\nc{\tet}{\tilde{\ast}} \nc{\teot}{\tilde{\odot}}
\nc{\oex}{\overline{x}} \nc{\oey}{\overline{y}}
\nc{\oez}{\overline{z}} \nc{\oea}{\overline{a}}
\nc{\oeb}{\overline{b}} \nc{\oef}{\overline{f}}
\nc{\weast}[1]{\widetilde{\ast}^{#1}}
\nc{\weodot}[1]{\widetilde{\odot}^{#1}} \nc{\hstar}[1]{\star^{#1}}
\nc{\lae}{\langle} \nc{\rae}{\rangle} \nc{\owd}{\overrightarrow{d}}
\nc{\owc}{\overrightarrow{c}} \nc{\mds}{\mathrm{MDS}}
\nc{\mda}{\mathrm{MDA}} \nc{\mdao}{\mathit{MDA}}

\nc{\CC}{\mathbb C}
\nc{\ZZ}{\mathbb{Z}}

\nc{\cala}{{\mathcal A}} \nc{\calb}{{\mathcal B}}
\nc{\calc}{{\mathcal C}}
\nc{\cald}{{\mathcal D}} \nc{\cale}{{\mathcal E}}
\nc{\calf}{{\mathcal F}} \nc{\calg}{{\mathcal G}}
\nc{\calh}{{\mathcal H}} \nc{\cali}{{\mathcal I}}
\nc{\call}{{\mathcal L}} \nc{\calm}{{\mathcal M}}
\nc{\caln}{{\mathcal N}} \nc{\calo}{{\mathcal O}}
\nc{\calp}{{\mathcal P}} \nc{\calr}{{\mathcal R}}
\nc{\cals}{{\mathcal S}} \nc{\calt}{{\mathcal T}}
\nc{\calw}{{\mathcal W}} \nc{\calk}{{\mathcal K}}
\nc{\calx}{{\mathcal X}} \nc{\CA}{\mathcal{A}}

\nc{\fraka}{{\mathfrak a}} \nc{\frakA}{{\mathfrak A}}
\nc{\frakb}{{\mathfrak b}} \nc{\frakB}{{\mathfrak B}}
\nc{\frakD}{{\mathfrak D}} \nc{\frakg}{{\mathfrak g}}
\nc{\frakH}{{\mathfrak H}} \nc{\frakL}{{\mathfrak L}}
\nc{\frakM}{{\mathfrak M}} \nc{\bfrakM}{\overline{\frakM}}
\nc{\frakm}{{\mathfrak m}} \nc{\frakP}{{\mathfrak P}}
\nc{\frakN}{{\mathfrak N}} \nc{\frakp}{{\mathfrak p}}
\nc{\frakS}{{\mathfrak S}}
\font\cyr=wncyr10 \font\cyrs=wncyr7

\nc{\yong}[1]{\textcolor{blue}{YZ: #1}}

\begin{document}

\title{Homotopy Transfer Theorem for linearly compatible di-algebras}
%

%

\author{Yong Zhang}
\address{Department of Mathematics, Zhejiang University, Hangzhou 310027, China}
\email{tangmeng@zju.edu.cn}


\date{\today}


\begin{abstract}
This paper studies the operad of linearly compatible di-algebras,
denoted by $As^{2}$, which is a nonsymmetric operad encoding the algebras with
two binary operations that satisfy individual and sum associativity
conditions. We also prove that the operad $As^{2}$ is exactly the
Koszul dual operad of the operad $^{2}As$ encoding totally compatible di-algebras. We show that
 the operads $As^{2}$ and $^{2}As$ are
Koszul by rewriting method. We make explicit the Homotopy Transfer
Theorem for $As^{2}$-algebras.
\end{abstract}


\maketitle

\tableofcontents

\setcounter{section}{0}

\section{Introduction}
The notion of associative algebra up to homotopy
has been introduced by Jim Stasheff in \cite{Js} under the name $A_{\infty}$-algebra.
It has the following
important property: starting with a differential graded associative (dga) algebra
$(A, d_{A})$, if $(V, d_{V})$ is a deformation retract of $(A, d_{A})$,
then $(V, d_{V})$ is not a
dga algebra in general, but an $A_{\infty}$-algebra. This is
Kadeishvili's theorem ~\cite{Tk}, which is also
called the Homotopy Transfer Theorem for associative algebras.




In this paper we are interested into replacing the associative operation on $A$ by two associative operations $x*y$ and $x\bullet y$ which are
\emph{linearly compatible}, that is, any linear combination of $*$ and $\bullet$ is associative.
We are going to determine the algebraic structure which is transferred to $V$.
It consists in $n$ $n$-ary operations, for any $n\geq 2$,
 which satisfy some relations analogous to the relations satisfied by the $n$-ary operations in an $A_{\infty}$-algebra.
 In \cite{LV} Loday and Vallette have given a generalization of the Homotopy Transfer Theorem for algebras over a Koszul operad,
 with explicit formulas. The key point is to make the Koszul dual cooperad explicit, and then to make also the differential map explicit
 in the cobar construction. First, we give a new proof of the theorem which says that the Koszul dual operad of linear compatible di-algebras
  is the operad of totally compatible di-algebras
  and that these two operads are Koszul.
   Second,  we describe explicitly the  composition in this operad $^{2}As$, so we obtain the relations satisfied by the generating operations.

Let $\bf k$ be a commutative unitary ring. The tensor product over $\bf k$ is denoted
 by $\ot_{\bf k}$ or simply by $\ot$ if it causes no confusion.


\section{Linearly compatible di-algebras}

We first recall the definition of linearly compatible di-algebra
introduced in \cite{St} by H.\ Strohmayer.
\begin{remark}{\rm
The operad of linearly compatible di-algebras is denoted by
$As^{\langle 2\rangle}$ in ~\cite{Zinb} and denoted
by $As^{2}$ in ~\cite{St},
in which Strohmayer has studied the compatible structures for various symmetric operads. From now on,
in our paper, we choose the symbol $As^{2}$ to denote the operad of linearly compatible di-algebras in nonsymmetric case.}

\end{remark}

\begin{defn}
{\rm A \emph{linearly compatible di-algebra} is defined to be a $\bf k$-module $V$ with two
binary operations $\opa$ and $\op$ that are associative and
satisfy
the relation
$$(x\op y)\opa z+(x\opa y)\op z=x\op(y\opa z)+x\opa(y\op z), \,\, \forall x,y,z \in V.$$
\mlabel{def:lca}}
\end{defn}

In \mcite{St}, we know that the operad
$As^{2}$ of linearly compatible di-algebras is the black product of $As$ and $Lie_2$.

We observe that these three relations are equivalent to the associativity relation for the operation $(x,y)\mapsto \lambda (x\op y)+ \mu(x\opa y)$ for any parameters $\lambda, \mu \in {\bf k}$.

\section{Totally compatible di-algebra and the operad $As^{2}$}
In this section, we give the definition of totally compatible di-algebra and describe its associated nonsymmetric (ns) operad $^{2}As$. We also show that the operad $^{2}As$
 is Koszul by the rewriting method.

\begin{defn}
{\rm ~\cite{St, ZBG} A \emph{totally compatible  di-algebra} is a $\bf k$-module
$A$ with two binary operations:
$$\opa,\bullet: A\ot A\rightarrow A,$$
satisfying the
$\text{TCD}$
axioms:
\begin{enumerate}
\item
$\opa$ and $\bullet$ are associative.
\item
$(x\opa y)\bullet z=x\opa(y\bullet z)=(x\bullet y)\opa z=x\bullet(y\opa z)$, $\forall x,y,z \in A$.
\end{enumerate}}\mlabel{def:tcda}
\end{defn}

Let $^{2}As$
denote
the operad of totally compatible di-algebras as that in symmetric case in \mcite{St}.
So we let $^{2}As$-algebra denote the totally compatible di-algebra.
From the definition of totally compatible di-algebra, we see that the operad $^{2}As$
is nonsymmetric, set-theoretic, binary and quadratic.
The structure of the ns operad $^{2}As$ can be derived from the following result.
\begin{prop}
The vector space $^{2}As_n$ is $n$-dimensional. Let $\mu_{ij}$ be the
operation, given by
$$\mu_{ij}(x_1 \cdots x_n)=x_1\opa \cdots \opa x_{i+1}\bullet \cdots \bullet x_n$$
with $i$ copies of $\opa$ and $j$ copies of $\bullet$.
Then the composition in the operad $^{2}As$ is given by
$$\gamma(\mu_{ij};\mu_{i_{1}\,j_{1}},\cdots,\mu_{i_{k}\,j_{k}})=\mu_{i+i_{1}+\cdots+i_{k}\,j+j_{1}+\cdots+j_{k}}.$$
In particular,
we have
$$\gamma(\mu_{10};\mu_{ij},\mu_{kl})=\mu_{i+j+1\,j+l};\quad \gamma(\mu_{01};\mu_{ij},\mu_{kl})=\mu_{i+k\,j+l+1}.$$
\mlabel{prop:tcd}
\end{prop}

\begin{proof}
In ~\cite{ZBG}, we
show
that the triple $(^{2}As(X):
=\overline{{\bf k}<X>}\ot {\bf k}<X>,
\opa,\op)$ is the free totally compatible dialgebra on the set $X$, showing
that the operad $^{2}As$ is $n$-dimensional in arity $n$.
Here ${\bf k}<X>$ denotes the free associative algebra over the set $X$, and $\overline{{\bf k}<X>}$ is its augmentation ideal.

Without loss of generality, for operations $\mu_{ij}$ and $\mu_{kl}$,
given any element $x_1 \cdots x_n$ in $^{2}As(X)$, we have

\begin{eqnarray*}
\gamma(\mu_{ij};\mu_{i_{1}\,j_{1}},\cdots,\mu_{i_{k}\,j_{k}})(x_1 \cdots x_n)
&=& \mu_{ij}(\mu_{i_{1}\,j_{1}}(x_1\cdots x_{i_1+j_1+1}),\cdots,\mu_{i_{k}\,j_{k}}(x_{i_{1}+\cdots+i_{k-1}+j_1+\cdots+j_{k-1}+1}\cdots x_n))\\
&=& \mu_{ij}(x\opa \cdots \opa x_{i_{1}+1}\op \cdots \op x_{i_{1}+j_{1}+1},\\
&\cdots,&  x_{i_{1}+\cdots+i_{k-1}+j_{1}+\cdots+j_{k-1}+1}\opa \cdots \opa x_{i_{1}+\cdots+i_{k-1}+j_{1}+\cdots+j_{k-1}+1+i_{k}}\op \cdots \op x_n)\\
&=& x_1\opa \cdots \opa x_{i_{1}+\cdots+i_{k}+i+1}\op \cdots \op x_n\\
&=& \mu_{i_1+\cdots+i_{k}+i\,j_1+\cdots+j_{k}+j}(x_1\cdots x_n)
\end{eqnarray*}
with $n-(i_{1}+\cdots+i_{k}+i+1)=j_1+\cdots+j_{k}+j$, implying
the composition $\gamma$ of $^{2}As$. Then we get the ns operad $^{2}As=(\bigoplus_n^{2}As_{n},\gamma)$.
\end{proof}

In ~\cite{St},
Strohmayer has proved that the operad $^{2}As$
is Koszul by
using the weight partition method.
Here we give a different proof based on rewriting systems.

\begin{theorem}
The operad $^{2}As$ is Koszul.
\end{theorem}
\begin{proof}
 Let $E$ be the generating space of binary operations with an ordered basis $\{\opa,\op\}$ such that $\opa >\op$.
 Let $\mu_1:=\opa$ and $\mu_2:=\op$.

Let $R$ be the space of relations, which is spanned
by a set of relators written as in the following basis by the definition
of totally compatible di-algebra in~\mref{def:tcda}:

Let $\circ_1:=((\cdot,\cdot),\cdot)$ and $\circ_2:=(\cdot,(\cdot,\cdot))$.
\begin{equation}
\mu_1\circ_1\mu_1-\mu_1\circ_2\mu_1=0,\mlabel{eq:(re)_1}
\end{equation}
\begin{equation}
\mu_1\circ_1\mu_2-\mu_1\circ_2\mu_2=0,\mlabel{eq:(re)_2}
\end{equation}
\begin{equation}
\mu_2\circ_1\mu_1-\mu_2\circ_2\mu_1=0, \mlabel{eq:(re)_3}
\end{equation}
\begin{equation}
\mu_2\circ_1\mu_2-\mu_2\circ_2\mu_2=0, \mlabel{eq:(re)_4}
\end{equation}
\begin{equation}
\mu_1\circ_2\mu_2-\mu_2\circ_2\mu_1=0. \mlabel{eq:(re)_5}
\end{equation}

Let the monomials $\mu_1\circ_1\mu_1,\,\mu_1\circ_1\mu_2,\,\mu_2\circ_1\mu_1,\,\mu_2\circ_1\mu_2,\,\mu_1\circ_2\mu_2$ be the leading terms of relations~${\mref{eq:(re)_1}-\mref{eq:(re)_5}}$, respectively.

Then the above choices provide rewriting rules of the form
\begin{align*}
\mu_i\circ_1\mu_j &\mapsto& \mu_i\circ_2\mu_j\\
\mu_1\circ_2\mu_2 &\mapsto& \mu_2\circ_2\mu_1\\
\text{leading\, terms} &\mapsto& \text{lower\, and\, non\, leading terms}
\end{align*}
with $i,j \in\{\ 1,2\}$, which give rise to the following critical
monomials
\[\left\{
\begin{array}{ll}
\mu_i\circ_1\mu_j\circ_1\mu_k, i,j,k\in\{1,2\} & \\
\qquad\qquad\qquad\qquad\qquad\qquad \text{if\,all\,the\,leading\,terms\,are\,in}
~{\mref{eq:(re)_1}-\mref{eq:(re)_4}};\\
\mu_1\circ_2\mu_2\circ_1\mu_2,\, \mu_1\circ_1\mu_1\circ_2\mu_2,\,
 \mu_2\circ_1\mu_1\circ_2\mu_2,\, \mu_1\circ_2\mu_2\circ_1\mu_1 &\\
 \qquad\qquad\qquad\qquad\qquad\qquad \text{if\,there\,leading\,terms\,are\,in}
~\mref{eq:(re)_5}.
\end{array}
\right.
\]\mlabel{eq:cri}

According to the rewriting method in \text{chapter 8} in~\cite{LV}, it is enough to
 check that all the critical monomials in~\mref{eq:cri}
 are confluent. We can see that relations ~${\mref{eq:(re)_1}-\mref{eq:(re)_4}}$
 are of associative type. We know that the critical monomials of associative type
 are confluent. Since their diamond is the following pentagon, see (figure 1).



$$\xymatrix{
 &(\mu_i\circ_1\mu_j)\circ_1\mu_k\cr
 \swarrow & &\searrow\cr
 (\mu_i\circ_2\mu_j)\circ_1\mu_k & & \mu_i\circ_1(\mu_j\circ_2\mu_k)\cr
\searrow & & \swarrow\cr
\mu_i\circ_2(\mu_j\circ_1\mu_k)& \longrightarrow & \mu_i\circ_2(\mu_j\circ_2\mu_k)
}$$\mlabel{figure 2}

 In order to prove that the critical monomials of relation~$\mref{eq:(re)_5}$ are confluent, we take the first one
  in Eq~$\mref{eq:(re)_5}$ as an example and the others can be proved in a similar way, see (figure 2).

 $$\xymatrix{
 & &\mu_1\circ_2\mu_2\circ_1\mu_2 && \cr
 \swarrow & & & &\searrow \cr
 (\mu_2\circ_2\mu_1)\circ_1\mu_2& & & & \mu_1\circ_2(\mu_2\circ_1\mu_2)\cr
 \downarrow & & & & \downarrow \cr\mu_2\circ_2(\mu_1\circ_1\mu_2)
 & & & & \mu_1\circ_2(\mu_2\circ_2\mu_2)\cr \downarrow & & & & \downarrow \cr
 \mu_2\circ_2(\mu_1\circ_2\mu_2)& & & & \mu_2\circ_2(\mu_1\circ_2\mu_2)\cr \searrow
 & &\mu_2\circ_2(\mu_2\circ_2\mu_1)& & \swarrow && \cr
}$$\mlabel{figure 2}

Since all the critical monomials are confluent,
the operad $^{2}As$ is a Koszul operad.
 \begin{remark}{\rm
 \begin{enumerate}
 \item
 The rewriting method is due to  E.\ Hoffbeck \cite{Eh}.
 \item
 The reader can find more details about rewriting method of ns operad in chapter $8$ of~\cite{LV}.
 \end{enumerate}}

 \end{remark}
\end{proof}

\begin{prop}
In \mcite{St}, the operad $As^{2}$ of linearly compatible di-algebras
is the Koszul dual operad
of the operad $^{2}As$.

\end{prop}

\begin{remark}
{\rm In \mcite{St}, this result is a special case in Prop1.7. Here we give a different proof
when considered all the operads being nonsymmetric.}
\end{remark}

\begin{proof}
Let $(\cdot_1,\cdot_2)_1$ denote
the operation which sends $(x,y,z)$ to
$((x\cdot_1 y)\cdot_2 z)$ and $(\cdot_1,\cdot_2)_2$
denote
the operation which sends $(x,y,z)$ to
$(x\cdot_1(y\cdot_2 z))$, with $\cdot_1,\cdot_2\in \{\opa,\op\}$.
The space $R$ of relations of $^{2}As$
 is determined by the relators
$$\left \{\begin{array}{l}
(\ast,\ast)_2-(\ast,\ast)_1\\
(\op,\op)_2-(\op,\op)_1\\
(\op,\ast)_2-(\op,\ast)_1\\
(\ast,\op)_2-(\ast,\op)_1\\
(\op,\ast)_2-(\ast,\op)_1
\end{array}
\right\}.$$ It is immediate to verify that its annihilator
$R^{\bot}$, with respect to the given product in chapter 7 in \cite{LV}, is the
subspace determined by the following relators
$$\left \{\begin{array}{l}
(\ast,\ast)_2-(\ast,\ast)_1\\
(\op,\op)_2-(\op,\op)_1\\
(\op,\ast)_2+(\ast,\op)_2-(\op,\ast)_1-(\ast,\op)_{2}.
\end{array}
\right.$$
These are precisely the expected relations in definition~\mref{def:lca}
\end{proof}

By \cite{GK} and \mcite{St}, since the binary quadratic operad $^{2}As$ is a Koszul operad,
it follows that its Koszul dual operad $As^{2}$ is also Koszul.

\section{Homotopy Transfer Theorem for
linearly compatible di-algebras.}

In this section, we make explicit the notion of $As^{2}$-algebra
up to homotopy i.e. ${As^{2}}_{\infty}$-algebra by describing the dg operad $As^{2}_{\infty}$.

Since the operad $As^{2}$ is Koszul, the dg operad $As^{2}_{\infty}$ is given by
$As^{2}_{\infty}:=\Omega((As^{2})^{\scriptstyle \textrm{!`}})$. So we need to
describe
 $(As^{2})^{\scriptstyle \textrm{!`}}={^{2}As}^{*}$, which is the co-operad linearly dual to
$^{2}As$. By proposition~\mref{prop:tcd}, we know that the space
${^{2}As}^{*}_{n}$ is $n$-dimensional. In order to describe the
differential of the cobar construction $As^{2}_{\infty}$, we need to introduce the following
definition and lemma.
\begin{defn}[$\text{chapter}$ 6 in~\cite{LV}]
{\rm
For any co-operad $(\mathcal{C},\Delta,\eta)$ with counit $\eta:\calc\rightarrow I$, we consider
the projection of the decomposition map to the infinitesimal part of the composite product
$\calc \circ \calc$. This map is called the \emph{infinitesimal decomposition map} of $\calc$
and is defined by the following composite
\[ \Delta_{(1)}:=\calc\xrightarrow{\Delta} \calc\circ\calc
\xrightarrow{\text{Id}_{\calc}\circ^{'}\text{Id}_{\calc}} \calc\circ(\calc;\calc)
\xrightarrow{\text{Id}_{\calc}\circ(\eta;\text{Id}_{\calc})}
\calc\circ(\mathrm{I};\calc)=\calc\circ_{(1)}\calc,
\] where the notation $\text{Id}_{\calc}\circ^{'}\text{Id}_{\calc}$ is the infinitesimal composite of $\text{Id}_{\calc}$
and $\text{Id}_{\calc}$ and the notation $\calc\circ_{(1)}\calc$ denotes the infinitesimal
composite of $\calc$ and $\calc$.
}
\end{defn}

\begin{remark}
{\rm $\Delta_{(1)}$ is also called the linear part of the co-composition $\Delta$.}
\end{remark}

\begin{lemma}
[\text{chapter} 6 in \cite{LV}] For a given co-operad $(\calc,\Delta,\eta)$, the dg cobar construction
of the co-operad $\calc$ is given by
$$\Omega(\calc):=(\calt(s^{-1}\overline{\calc}),d),$$ with $d$ induced by $\Delta_{(1)}$
as follows
\[
d: {\bf k}s^{-1}\ot \overline{\calc}\xrightarrow{\Delta_{s}\ot \Delta_{(1)}} ({\bf k}s^{-1}\ot
{\bf k}s^{-1})\ot(\calc\circ_{(1)}\calc)\xrightarrow{\text{Id}\ot\tau\ot\text{Id}}
({\bf k}s^{-1}\ot \overline{\calc})\circ_{(1)}({\bf k}s^{-1}\ot \overline{\calc})
\cong \mathcal{T}(s^{-1}\overline{\calc})^{(2)}\rightarrowtail{} \mathcal{T}(s^{-1}\overline{\calc}),
\]
where $s$ is the decoration, modifying the degree of the objects in $\overline{\calc}$.\mlabel{lemma:cbd}

\end{lemma}

From lemma~\mref{lemma:cbd}, it is sufficient to make explicit the
infinitesimal part $\Delta_{(1)}$ in the operad ${^{2}As}^{*}$ to
get the differential map $\partial$ of cobar construction
$As^{2}_\infty=\calt((As^{2})^{\scriptstyle \textrm{!`}})$, without decoration $s$.

\begin{theorem}
The linear part of the co-composition $\Delta$ in ${^{2}As}^{*}$ is given by
$$\Delta_{(1)}(\mu_{c\,d}^{c})=\sum_{\substack{c=i+a\\
d=j+b}}
\mu_{i\,j}^{c}(\underbrace{id,\cdots,id}_{p\,\,\text{copies}},
\mu_{a\,b}^{c},\underbrace{id,\cdots,id}_{r\,\,\text{copies}})$$ excluding $(a,b)=(i,j)=(0,0)$ with
$\mu_{i\,j}\in {^{2}As}_n, i+j=n-1$.\mlabel{thm:tcd}
\end{theorem}

\begin{proof}
By the property of linearly dual basis, it is a straightforward computation.
Since $$\gamma(\mu_{i\,j};\underbrace{id,\cdots,id}_{p\,\, \text{copies}},
\mu_{a,b},\underbrace{id,\cdots,id}_{r\,\, \text{copies}})=\mu_{i+a\,j+b},$$
then
$$\Delta_{(1)}(\mu_{c\,d}^{c})=\sum_{c=i+a,\,d=j+b}\,\,
\mu_{i\,j}^{c}(\underbrace{id,\cdots,id}_{p\,\,\text{copies}},
\mu_{a\,b}^{c},\underbrace{id,\cdots,id}_{r\,\,\text{copies}}).$$
\end{proof}

Let $m_{i\,j}:=\mu_{i\,j}^{c}$ be the generator of the cobar construction $As^{2}_{\infty}$. Then we
get the following result.

\begin{theorem}
The operad $As^{2}_{\infty}$ is generated by the operations $m_{i\,j}$,
 with $|m_{i\,j}|=n-2$ \,for\, $i+j=n-1, i,j\geq 0$, which satisfy the following formula:
 $$\partial(m_{i\,j})=\sum_{\substack{a+c=i,\,b+d=j\\q=c+d+1,\,p+q+r=i+j+1}}(-1)^{p+qr}m_{a\,b}
 (\underbrace{id,\cdots,id}_{p\,\,\text{copies}},m_{c\,d},\underbrace{id,\cdots,id}_{r\,\,\text{copies}}).$$
\end{theorem}

\begin{proof}
From the definition of $As^{2}_{\infty}$ and Theorem~\mref{thm:tcd},
it follows that  $As^{2}_{\infty}$ is generated by the operations $m_{i\,j}$.

By definition of the cobar construction in lemma~\mref{lemma:cbd}, the boundary
map on $\Omega((As^{2})^{\scriptstyle \textrm{!`}})$ is induced by the co-operad structure
of $(As^{(2)})^{\scriptstyle \textrm{!`}}$, and more precisely by $\Delta_{(1)}$ of linear
dual co-operad ${^{2}As}^{*}$ given by Theorem ~\mref{thm:tcd} as:
$$\Delta_{(1)}(\mu_{c\,d}^{c})=\sum_{\substack{c=i+a\\
d=j+b}}
\mu_{i\,j}^{c}(\underbrace{id,\cdots,id}_{p\,\,\text{copies}},
\mu_{a\,b}^{c},\underbrace{id,\cdots,id}_{r\,\,\text{copies}})$$ excluding $(a,b)=(i,j)=(0,0)$ with
$\mu_{i\,j}\in {^{2}As}_n, i+j=n-1$

By the construction of the differential given in lemma~\mref{lemma:cbd}, we have
 $$\partial(m_{i\,j})=\sum_{\substack{a+c=i,\,b+d=j\\q=c+d+1,\,p+q+r=i+j+1}}(-1)^{p+qr}m_{a\,b}
 (\underbrace{id,\cdots,id}_{p\,\,\text{copies}},m_{c\,d},\underbrace{id,\cdots,id}_{r\,\,\text{copies}}).$$

The signs are obtained by comparison with the dg operad $A_{\infty}$.
\end{proof}

From the above results, we get the Homotopy Transfer Theorem
for the operad $As^{2}$.

\begin{theorem}
Let

$$\xymatrix{
*{\qquad ({A,d_{A}})} \ar@(ul,dl)[]_{h} \ar@/^/[rr]^p
&& *{(V,d_{V})} \ar@/^/[ll]^i}$$
$$i=\text{quasi-isomorphism} \qquad\text{Id}_{A}-ip=d_{A}h+hd_{A},$$
 be a deformation retract. If $(A,d_{A})$ is a dg $As^{2}$-algebra,
  then $(V,d_{V})$ inherits an ${As^{2}}_{\infty}$-algebra structure
   $\{m_{ij}\}_{i+j=n-1}$ with $n\geq 2$,
  which extends functorially the binary operations of $A$.

\end{theorem}

\begin{proof}
The conclusion is a direct consequence of the Homotopy Transfer
Theorem given in \text{chapter 10} in ~\cite{LV} applied to the  Koszul
operad $As^{2}$.

\end{proof}

$\bf{Acknowledgements}$ I am greatly indebted to my supervisors
professor Jean-Louis Loday (CNRS and Strasbourg University) and
professor Fang Li (Zhejiang University) for their constant support.
Part of this work was achieved in Strasbourg, France, during a visit sponsored by the Centre National de la Recherche Scientifique.

\end{document}